\documentclass[11pt,a4paper, envcountsame]{llncs}

 \usepackage[colorlinks,citecolor=blue,urlcolor=blue, linkcolor=blue]{hyperref}

 \usepackage{amsmath, amssymb, xspace, mathabx}
 \usepackage{graphicx}

 \usepackage[all]{xy}
\usepackage[colorinlistoftodos]{todonotes}

\newtheorem{thm}[theorem]{Theorem}
\newtheorem{prop}[theorem]{Proposition}
\newcommand{\cost}{\mathbf{c}}
\newcommand{\dost}{\mathbf{d}}

\newcommand{\limcost}{\underline{\cost}}

\newcommand{\DII}{\Delta^0_2}
\newcommand{\NN}{{\mathbb{N}}}
\newcommand{\RR}{{\mathbb{R}}}
\newcommand{\R}{{\mathbb{R}}}

\newcommand{\sub}{\subseteq}
\newcommand{\sN}[1]{_{#1\in \omega}}

\newcommand{\ML}{Martin-L{\"o}f}
\newcommand{\SI}[1]{\Sigma^0_{#1}}
\newcommand{\PI}[1]{\Pi^0_{#1}}
\newcommand{\PPI}{\PI{1}}

\newcommand{\bi}{\begin{itemize}}
\newcommand{\ei}{\end{itemize}}
\newcommand{\bc}{\begin{center}}
\newcommand{\ec}{\end{center}}

\newcommand{\Halt}{{\ES'}}
\newcommand{\ES}{\emptyset}

\newcommand{\ria}{\rightarrow}
\newcommand{\tp}[1]{2^{#1}}
\newcommand{\ex}{\exists}
\newcommand{\fa}{\forall}

\newcommand{\seqcantor}{2^\NN}

\newcommand{\cantor}{\seqcantor}

\newcommand{\leT}{\le_{\mathrm{T}}}

\newcommand{\MLR}{\mbox{\rm \textsf{MLR}}}
\newcommand{\Weaktwo}{\mbox{\rm \textsf{W2R}}}
\newcommand{\CR}{\mbox{\rm \textsf{CR}}}
\newcommand{\Low}{\mbox{\rm \textsf{Low}}}

\newcommand{\n}{\noindent}

\newcommand{\leb}{\mathbf{\lambda}}

\newcommand{\frd}{\mathfrak{d}}

\newcommand \seq[1]{{\left\langle{#1}\right\rangle}}

\newcommand\+[1]{\mathcal{#1}}
\newcommand{\sC}{\+ C}

\newcommand{\ol}{\overline}
\newcommand{\ul}{\underline}

\newcommand{\lra}{\leftrightarrow}

\newcommand{\RA}{\Rightarrow}

\begin{document}
\title{Lowness, randomness, and computable analysis}
\author{Andr\'e Nies}
 \institute{Department of Computer Science,  
University of Auckland}

\maketitle

 \begin{abstract} Analytic concepts  contribute to  our understanding of randomness of  reals via algorithmic tests. They also  influence the interplay between randomness   and lowness notions. We provide a survey.  \end{abstract}
\section{Introduction}  Our basic objects of study are infinite bit sequences, identified with sets of natural numbers, and often simply called sets. A  lowness notion  provides a sense in which a set $A$  is close to computable. For example, $A$ is \emph{computably dominated} if each function computed by A is dominated by  a computable function; $A$ is \emph{low} if the halting problem relative to $A$  has   the least possible Turing complexity, namely $A' \equiv_T \Halt$.  These two notions are incompatible outside the computable sets, because every non-computable $\DII$ set has hyperimmune degree. 

Lowness notions have been studied for at least 50 years  \cite{Spector:56,Martin.Miller:68,Jockusch.Soare:72}.
More recently, and perhaps suprisingly, ideas motivated by the intuitive notion of randomness have  been  applied  to the  investigation of   lowness. On the one hand, these ideas have  led to new lowness notions. For instance, $K$-triviality of a set of natural numbers (i.e., being far from random in a specific sense) coincides with lowness for \ML\ randomness, and many other notions. 				
On the other hand, they have  been applied   towards a deeper understanding of previously known lowness notions. Randomness   led to the study of  an important subclass of the computably dominated sets,   the computably traceable sets~\cite{Terwijn.Zambella:01}.	Superlowness of an oracle $A$, first studied by Mohrherr~\cite{Mohrherr:86},  says that $A' \equiv_{tt} \Halt$; despite the fact that   the low basis theorem \cite{Jockusch.Soare:72} actually yields  superlow 	sets, the importance of superlowness  was not fully appreciated until the investigations of lowness via randomness. For instance,  every $K$-trivial  set is superlow~\cite{Nies:AM}. 
				
				Computable analysis allows us to characterise  several    randomness notions that were  originally defined in terms of algorithmic tests. Schnorr~\cite{Schnorr:75} introduced two randomness notions for a bit sequence $Z$  via    the failure of  effective betting strategies. Nowadays they are called computable  randomness and  Schnorr randomness. Computable randomness says that no effective betting strategy (martingale) succeeds on $Z$, Schnorr randomness that no such strategy succeeds quickly (see \cite{Downey.Hirschfeldt:book,Nies:book} for background). Pathak~\cite{Pathak:09}, followed by Pathak, Rojas and Simpson~\cite{Pathak.Rojas.ea:12} characterised Schnorr randomness: $Z$ is Schnorr random iff an effective version of the Lebesgue differentiation theorem holds at the    real $z\in [0,1]$ with binary expansion~$Z$.   Brattka, Miller and Nies~\cite{Brattka.Miller.ea:16}  showed that   $Z$  is computably random if and only if every nondecreasing computable function is differentiable at  $z$.	 See Section~\ref{s:analysis}		for detail.		
					 	
				From    2011 onwards, concepts from analysis have also influenced    the interplay of lowness and randomness. The Lebesgue density theorem for effectively closed sets  $\+ C$  provides two   randomness notions  for a bit sequence $Z$  which are slightly stronger than \ML's. In the stronger form, the density of any such set $\+ C$  that contains $Z$  has to  be 1 at $Z$; in the weak form, the density is merely required to be  positive. One has to  require  separately  that $Z$ is ML-random  (even the stronger   density condition doesn't imply this, for instance because a 1-generic also satisfies this   condition). These two notions have  been used to obtain a Turing incomplete \ML\ random above all the K-trivials, thereby    solving the so-called ML-covering problem. We give more detail in Section~\ref{s:LDT-K};  also see the survey~\cite{Bienvenu.Day.ea:14}.

In current research, concepts inspired by analysis are used to stratify lowness notions. Cost functions describe  a dense hierarchy of subideals of the K-trivial Turing degrees (Section~\ref{s:costfunctions}).   The Gamma   and Delta parameters are  real numbers assigned to   a Turing degree. They provide a coarse   measure of  its complexity in terms of the asymptotic density of bit sequences (Section~\ref{s:gamma_delta}).  

The present paper surveys the study of lowness notions via randomness. In  Sections~\ref{s:early} and~\ref{s:redness} we   elaborate on the   background in lowness, and how it was influenced by randomness. Section~\ref{s:analysis} traces the interaction   of computable analysis and randomness from Lebesgue to the present.   Section~\ref{s:LDT-K} shows how some of the advances via computable analysis aided the understanding of lowness through randomness. Sections \ref{s:costfunctions} and~\ref{s:gamma_delta}   dwell on  the   most recent developments. A final section contains open questions. % and hitherto unexplored directions. 
 
 \section{Early days of lowness} 
 \label{s:early}
 Spector~\cite{Spector:56} was the first to construct a  Turing degree that is minimal among the nonzero degrees. Sacks~\cite{Sacks:61} showed that such a degree can be  $\DII$. Following these results, as well as  the Friedberg-Muchnik theorem and   Sacks' result~\cite{Sacks:64} that the c.e.\ Turing degrees are dense, the interest of early computability theorists focussed on relative complexity of sets:  comparing their complexity via an appropriate reducibility. Absolute complexity, which means finding  natural lowness classes and studying membership in them,  received somewhat less attention, and was mostly restricted to classes defined via the Turing jump. Martin and Miller~\cite{Martin.Miller:68} built a perfect closed set of computably dominated oracles. Jockusch and Soare \cite{Jockusch.Soare:72} proved that every non-empty effectively closed set contains a low oracle. These constructions used recursion-theoretic versions of forcing. Jockusch   published  papers such as~\cite{Jockusch:68,Jockusch:69a} that explored notions such as  degrees of   diagonally noncomputable functions, and degrees of bi-immune sets.  
 Downey's     work in the   1980s  was    important for the development of our understanding of lowness. For instance, Downey and Jockusch~\cite{Downey.Jockusch:87}  studied complexity of sets, both relative and absolute,  using ever more sophisticated methods.  
 
\section{Randomness interacts with lowness} \label{s:redness}
We begin with the following randomness notions:
\bc weakly  2-random $\rightarrow$ ML-random  $\rightarrow$ computably rd. $\rightarrow$ Schnorr rd.  \ec
  $Z$ is weakly  2-random iff $Z$ is  in no null $\PI 2$ class.  Section~\ref{s:LDT-K} will develop  notions implied by  weak 2-randomness, and somewhat  stronger than ML-randomness.

Lowness can be used to understand randomness via the randomness enhancement principle~\cite{Nies:ICM}, which says that   sets already enjoying a randomness property get more random as they become  lower in the sense of computational complexity. Every non-high Schnorr random is ML-random. A ML-random is weakly 2-random iff it forms a Turing minimal pair with $\ES'$. See~\cite{Downey.Hirschfeldt:book,Nies:book}.

Here we are mostly interested in the converse interaction:  studying lowness via randomness. 
Let $K(x)$ denote the prefix free version of Kolmogorov complexity of  a binary string $x$. The $K$-trivial sets  were introduced by Chaitin~\cite{Chaitin:76} and   studied by Solovay in an unpublished manuscript~\cite{Solovay:75}, rediscovered  by Calude in the 1990s. Most of this manuscript  is    covered in  Downey and Hirschfeldt's  monumental work~\cite{Downey.Hirschfeldt:book}.   We say that $A$ is $K$-trivial if $\ex b \fa n \, K(A\mid n) \le K(n) +b$. By the Levin-Schnorr characterisation, $Z$ is ML-random iff $\ex d \fa n \, K(Z\mid n) \ge n-d$. Since $K(n) \le \log_2 n+ O(1)$, this definition  says  that $K$-trivials are far from random. Each computable set is $K$-trivial; Solovay built a  $K$-trivial $\DII$ set $A$ that is not computable. This was later improved to a c.e.\ set $A$ by Downey et al. \cite{Downey.Hirschfeldt.ea:03}, who used what became later  known as  a cost function construction.
 
An oracle $A$ is called \emph{low for a randomness notion} $\+ C$ if every $\+ C$-random set is already $\+ C$-random relative to $A$. $K$-triviality appears to be  the universal lowness class for randomness notions based on c.e.\ test notions.   $A$ is $K$-trivial iff  $A \in \Low(\Weaktwo,\CR) $, namely every weakly 2-random set is computably random relative to~$A$. This was shown by the author~\cite[8.3.14]{Nies:book} extending the result that $ \Low(\Weaktwo,\MLR) $ coincides with $K$-triviality \cite{Downey.Nies.ea:06}. As a consequence,  for any  randomness   notion $\+ D$ in between weak-2 randomness and computable randomness, lowness for $\+ D$ implies $K$-triviality. For many notions, e.g.\ weak-2 randomness \cite{Downey.Nies.ea:06} and ML-randomness~\cite{Nies:AM}, the classes actually coincide.

Some of  the $K$-trivial story, including the   roles Downey, Hirschfeldt and the author  have played in it, is   vividly  described in  \cite{Nies:5questions}. Background and  more detailed  coverage for the developments up to 2009 can be found in   the aforementioned books \cite{Downey.Hirschfeldt:book,Nies:book}.
 
 \section{Randomness and computable analysis} \label{s:analysis}
 We discuss the influence exerted by computable analysis    on the study of randomness notions. Thereafter we will return to our main topic, lowness notions.
 
  Analysis and ergodic theory have plenty of theorems saying that a property of being well-behaved holds at almost every point.  Lebesgue proved that a function of bounded variation  defined on the unit interval  is differentiable at almost every point.  He also proved the density theorem, and the stronger differentiation theorem, that now both bear his name. The density theorem says that   a measurable set $\+ C \sub [0,1]$ has density one at  almost every of its members $z$. To have density one at $z$ means intuitively that many points close to $z$ are also in $\+ C$, and this becomes more and more apparent as one ``zooms in" on $z$: 
 \begin{definition} \label{def:LD} Let $\lambda$ denote Lebesgue measure on $\R$.
We define the lower Lebesgue density of a set $\sC \subseteq \R$ at a point~$z$ to be the limit inferior $$\ul \varrho(\sC | z):=\liminf_{|Q| \to 0} \frac{\lambda(Q \cap \sC)}{|Q|},$$
where $Q$ ranges over open intervals containing $z$.
The Lebesgue density of $\sC$ at $z$ is the limit (which may not exist)
$$ \varrho(\sC | z):=\lim_{|Q| \to 0} \frac{\lambda(Q \cap \sC)}{|Q|}.$$

\end{definition} 
%By standard arguments one can assume that  the intervals $Q$ are centred at $z$.
 Note that  $0 \le \ul \varrho(\sC | z) \le 1$. 
  \begin{theorem}[Lebesgue \cite{Lebesgue:1909}] 
Let $\sC \subseteq \R$ be a measurable set. We have   $\ul \varrho(\sC | z)=1$ for almost every $z\in \sC$.
\end{theorem}
The  Lebesgue differentiation theorem says that for almost every $z$,  the value of an  integrable function $g$ at $z\in [0,1]$ is approximated by the average of the values around $z$, as one zooms in on~$z$. 
  A point $z$ in the domain of $g$ is called a \emph{weak Lebesgue point} of $g$ if
\begin{equation} \label{eqn:density}  \lim_{\leb Q \to 0}  \frac{1}{\lambda(Q)}\int_{Q} g d\lambda \end{equation} exists, where $Q$ ranges over open intervals  containing $z$; 
 we call~$z$  a \emph{Lebesgue point}  of $g$  if this value equals $g(z)$.  
  \begin{thm}[Lebesgue \cite{Lebesgue:1904}] \label{th:LDT} Suppose $g$ is an integrable function on $[0,1]$. Then almost every $z\in [0,1]$ is a Lebesgue point of $g$.\end{thm} 
  Lebesgue   \cite{Lebesgue:1910}  extended this result to higher dimensions, where $Q$ now ranges over open cubes containing~$z$.
  Note that if $g $ is the characteristic function of a measurable set $\sC$, then the expression (\ref{eqn:density}) is precisely the density of $\sC$ at $z$.

 In ergodic theory, one of the best-known ``almost everywhere" theorems is due to G.\ Birkhoff: intuitively, given an   integrable function $g$ on a  probability space with a measure preserving operator $T$,  for almost every point  $z$,  the average of $g$-values at iterates, that is,  $\frac 1 n \sum_{i< n} g(T^i(z))$, converges. If $T$ is ergodic (i.e., all $T$-invariant measurable  sets are null or co-null), the limit equals  $\int g$; in general the limit  is given by the conditional expectation of~$g$ with respect to the $\sigma$-algebra of $T$-invariant sets. 
 
The important insight is this: if the collection of given objects in  an a.e.\ theorem  is  effective in some particular sense, then the theorem  describes a randomness notion  via algorithmic tests. Every   collection of effective objects constitutes a test, and failing it means to be a point of exception for this collection.  Demuth~\cite{Demuth:75} (see below)  made this  connection    in the setting of constructive mathematics. In the usual classical setting,  V'yugin~\cite{Vyugin:98}    showed  that \ML\ randomness of a point $z$ in a computable probability space suffices for the existence of the  limit in Birkhoff's theorem when $T$ is computable and $g$ is $L_1$-computable.  Here $L_1$-computability means in essence that the function can be effectively approximated by step functions, where the distance is measured in the usual  $L_1$-norm. About ten years later, Pathak \cite{Pathak:09}  showed that ML-randomness of $z$   suffices for the existence of the limit in the Lebesgue differentiation theorem when the given function $f$ is $L_1$-computable. This works even when $f$ is defined on $[0,1]^n$ for some $n>1$. Pathak, Rojas and Simpson \cite{Pathak.Rojas.ea:12} showed  that in fact the weaker condition of Schnorr randomness on $z$ suffices. They also showed a converse:   if $z$ is not Schnorr random, then for some $L_1$-computable function~$f$ the limit fails to exist.  Thus, the Lebesgue differentiation theorem, in this effective setting, characterises Schnorr randomness. This converse   was obtained independently by Freer et al. \cite{Freer.Kjos.ea:14}, who also extended the    characterisation of Schnorr randomness to   $L_p$-computable functions,  for any fixed computable real $p \ge 1$. 

In the meantime, 
 Brattka, Miller and Nies proved the above-mentioned effective form of Lebesgue's  theorem~\cite{Lebesgue:1909} that  each nondecreasing function is  a.e.\ differentiable: each nondecreasing \emph{computable} function is differentiable at every computably random real  (\cite{Brattka.Miller.ea:16}, the  work for which was carried out from late 2009).  Later on, Nies~\cite{Nies:14} gave a different, and  somewhat simpler, argument for this result involving the geometric notion of porosity. With some extra complications the latter  argument   carries over to  the   setting  of polynomial time  computability, which was the main thrust of~\cite{Nies:14}. 
 
 Jordan's decomposition theorem says that for every function $f$ of bounded variation there are  nondecreasing functions $g_0, g_1$ such that $f = g_0-g_1$. This is almost trivial in the setting of analysis (take $g_0(x)$ to be the variation of $f$ restricted to  $[0,x]$, and let $g_1 = f-g_0$).  Thus   every bounded variation function $f$  is a.e.\ differentiable. For computable $f$, it turns out that ML-randomness of~$z$  may be  required  to ensure that $f'(z)$ exists; the reason is that the two functions obtained by the Jordan decomposition cannot   always be chosen to be computable.   Demuth~\cite{Demuth:75} had obtained results in the constructive setting which, when re-interpreted classically, show that $z$ is ML-random iff every computable function of bounded variation is differentiable at $z$.  Brattka et al.\  \cite{Brattka.Miller.ea:16} gave   alternative proofs of both implications. For the harder implication, from left to right, they used  their main result on computable nondecreasing functions, together with  the fact that the possible Jordan decompositions of a computable  bounded variation function form a $\PI 1$ class, which therefore has a   member in which  $z$ is random. See the recent survey~\cite{Kucera.Nies:ea:15} for more on Demuth's work   as an early example of an  interplay   between randomness and computability.
 
\section{Lebesgue density and $K$-triviality} 
\label{s:LDT-K}
Can analytic notions   aid in the   study of lowness via randomness?  The answer is ``yes, but only indirectly". Analytic notions help because they  bear on our view of randomness. In this section we will review how the   notion of Lebesgue density helped solving  the ML-covering problem, originally asked by Stephan (2004). This was one of five ``big"  questions in~\cite{Miller.Nies:06}. Every c.e.\ set below a Turing incomplete random is a base for ML-randomness, and hence $K$-trivial \cite{Hirschfeldt.Nies.ea:07}. The covering problem asks whether the converse holds: is every c.e.\ $K$-trivial $A$ below an incomplete random set?  Since every $K$-trivial is Turing below a c.e.\ $K$-trivial \cite{Nies:AM}, we can as well omit the hypothesis that $A$ be c.e.

\subsection*{Computable analysis $\rightsquigarrow$ randomness}
\label{ss:DevAn}
Some effective versions of almost everywhere theorems lack a predefined randomness notion. In the context of Theorem~\ref{th:LDT}, the statement that almost every point is a \emph{weak} Lebesgue point will be called the weak Lebesgue differentiation theorem. We have already discussed the fact  that the weak Lebesgue differentiation theorem  for $L_1$-computable functions characterises  Schnorr randomness. A function  $g$ is lower semicomputable if $\{x\colon g(x) > q \}$ is $\SI 1$ uniformly in a rational $q$, and upper semicomputable if $\{x\colon g(x) <  q \}$ is $\SI 1$ uniformly in a rational $q$.  Which degree of  randomness does a point $z$  need in order  to   ensure  that $z$ is a (weak) Lebesgue point for all  lower (or equivalently, all upper) semicomputable functions? 

Even ML-randomness is insufficient for this. Let $z= \Omega$ denote Chaitin's halting probability, and consider the  $\PPI$-set $\sC = [\Omega, 1]$.  The real $z$ is ML-random, and in particular normal: every block of bits of length $k$ (such as $110011$) occurs with limiting frequency $\tp{-k}$ in its binary expansion.  Suppose $ z \in Q$ where $Q = (i \tp{-n} , (i+1) \tp{-n})$ for some $i< n$. If the binary expansion of $z$ has a long block of 0s from position $n$ on, then $ {\lambda(Q  \cap \sC)}/|Q|$ is large; if $z$ has a long block of 1s from $n$ on then it is small. This implies that $ \lambda(Q)/|Q|$ oscillates between values close to $0$ and close to $1$ as $Q$ ranges over smaller and smaller basic dyadic intervals containing~$z$.  So it cannot be the case that $\ul \varrho(\sC | z)=1$;  in fact the density of $\sC$ at $z$ does not exist. This means that $z$ is not a weak Lebesgue point for the upper semicomputable function  $1_\sC$.  

We say that a ML-random real $z$ 
is \emph{density random} if $\ul \rho(\+ C  \mid z) =1$ for each $\PPI$ set $\+ C$ containing~$z$.
Several equivalent characterisations of density randomness are given in \cite[Thm. 5.8]{Miyabe.Nies.Zhang:15}; for instance, a real  $z$ is density random iff $z$ is a weak Lebesgue point of each lower semicomputable function on $[0,1]$ with finite integral, iff $z$ is  a full Lebesgue point of each such  function.

 \subsection*{Randomness $\rightsquigarrow $ lowness}
 The approach of the Oberwolfach group (2012) was mostly within the classical interplay of randomness and computability. Inspired by the notion of balanced randomness introduced in \cite{Figueira.Miller.ea:09}, they defined a new   notion, now called Oberwolfach (OW) randomness~\cite{Bienvenu.Greenberg.ea:16}. A test  notion equivalent to Oberwolfach tests, and easier to use,  is as  follows. A descending uniformly $\SI 1$ sequence  of sets $\seq{ G_m}\sN m$, together with a left-c.e.\ real $\beta$ with a computable approximation $\beta= \sup_s \beta_s$,  form  a \emph{left-c.e.} test if 
 $\leb G_m = O(\beta-\beta_m)$ for each $m$. Just like in the original definition of Oberwolfach tests, the test components cohere. If there is  an increase $\beta_{s+1} - \beta_s = \gamma >0$,  then all components $G_m$ for $m < s$ are allowed to add up to $\gamma$ in measure, as long as the sequence remains descending. We think of first $G_0$ adding some portion of measure of at most $\gamma$, then $G_1$ adding some portion of that, then $G_2$  a portion of that second portion, and so on up to $G_s$.

 The Oberwolfach group~\cite[Thm.\ 1.1]{Bienvenu.Greenberg.ea:16}  showed that if $Z$ is ML-random, but not OW-random, then $Z$ computes each $K$-trivial. They also proved that OW-randomness implies density randomness.

 \subsection*{Analysis $\rightsquigarrow $ randomness $\rightsquigarrow $ lowness}
Often the notions of density are studied in  the context of Cantor space $\cantor$,  which is easier to work with than  the  unit interval.  In this  context one  defines the density at a a bit sequence $Z$ using  basic dyadic intervals that are given by longer and longer  initial segments of  $Z$. In the context of randomness this turns out to be  a minor change. If $z$ is a ML-random real  and $Z$ its binary expansion, then each $\PPI$ set  $\+ C \sub [0,1]$ has positive density at $z$ iff    each $\PPI$ set $\+ C \sub \cantor$ has positive density at $Z$ iff $Z$ is Turing incomplete,  by a result in Bienvenu et al.~\cite{Bienvenu.Hoelzl.ea:12a}. Dyadic and full density~1 also coincide for ML-random reals by a  result of Khan and Miller~\cite[Thm. 3.12]{Khan:15}.

 Day and Miller  \cite{Day.Miller:15}  used a forcing partial order specially adapted to the setting of intermediate density to prove that there is a ML-random $Z$ such that $\ul \rho (\sC \mid Z) >0$ for each $\PPI$ class $\sC \sub \cantor$, and at the same time  there is a $\PPI$ class $\+ D \ni Z$ such that  
 $\ul \rho (\+ D \mid Z) < 1$.  Hence $Z$ is incomplete ML-random and not Oberwolfach random. By the aforementioned result of the Oberwolfach group~\cite[Thm.\ 1.1]{Bienvenu.Greenberg.ea:16} this means that the single oracle $Z$ computes each $K$-trivial, thereby giving a strong affirmative answer to the covering problem.
 
  Day and Miller also refined their argument in order  to make $Z$ a $\DII$ set. 
 No direct construction is known to build a $\DII$ incomplete ML-random that is not Oberwolfach random. In fact, it is open whether Oberwolfach and density randomness coincide (see Question~\ref{qu:OW} below).

\section{Cost functions and subclasses of the $K$-trivials} 
\label{s:costfunctions}

In this and the following section, we survey ways  to gauge the complexity of Turing degrees directly with    methods inspired by analysis. The first method  only  applies to $K$-trivials:   use the analytical tool of a cost function to study proper subideals of the ideal of $K$-trivial Turing degrees. This yields  a dense hierarchy of ideals parameterised by rationals in $(0,1)$.

The second method  assigns real number parameters  $\Gamma(\mathbf a) , \Delta(\mathbf a) \in [0,1]$ to Turing degrees $\mathbf a$ in order to measure their complexity.  These  assignments can be interpreted in the context of   Hausdorff distance in pseudometric spaces. In a sense,  this second attempt turns out to be too coarse  because in both variants, only the  values $0$ and $1/2$ are possible for  non-computable Turing degrees (after a modification of the definition which we also present, the classes of sets with value $0$ have    subclasses that are potentially proper).  However, this also shows that these few  classes of complexity obtained must be natural and  important. Although they richly interact with previously studied classes, they haven't as yet been fully  characterised by other means. 

Both approaches are     connected to randomness through the investigations of the concepts, rather than   directly  through the definitions. We will explain these connections as we go along.

 \subsection*{Cost functions}
 
 Somewhat extending \cite[Section 5.3]{Nies:book}, we say that a \emph{cost function} is a computable function
 \[ \cost\colon  \NN \times \NN \ria \{x \in \RR \,:\,  x \ge 0\}.
 \]
For  background on cost functions see \cite{Nies:CalculusOfCostFunctions}.
 We say that~$\cost$ is \emph{monotonic} if $\cost(x+1,s) \le \cost(x,s) \le \cost(x,s+1)$ for each~$x$ and~$s$; we also assume that $\cost(x,s)=0$ for all $x\ge s$.  We view $\cost(x,s)$ as the cost of changing at stage~$s$ a guess $A_{s-1}(x)$ at the value $A(x)$,  for some $\Delta^0_2$ set~$A$. Monotonicity means that the cost of a change increases with time, and that  changing the guess at a smaller number is  more costly. 

If $\cost$ is a cost function,  we let $\limcost(x) = \sup_s \cost(x,s)$. To be useful, a monotonic cost function $\cost$ needs to satisfy the \emph{limit condition}:  $\limcost(x)$ is finite for all~$x$ and $\lim_x \limcost(x)= 0$.  
  \begin{definition}[\cite{Nies:book}] \label{def:obeying_a_cost_function}
Let $\seq{A_s}$ be a computable approximation of a $\Delta^0_2$ set~$A$, and let~$\cost$ be a cost function. The \emph{total $\cost$-cost} of the approximation is 
\[   \cost  (\seq{A_s}\sN s) = \sum_{s\in\omega} \left\{  \cost(x,s) \,: \,    x  \text{ is least such that } A_{s-1}(x) \ne A_{s}(x) \right\}.\]
We say that a $\Delta^0_2$ set~$A$ \emph{obeys}~$\cost$ if the total $\cost$-cost of \emph{some} computable approximation of~$A$ is finite. We write $A \models \cost$.
\end{definition}
\begin{figure}[htb] \bc \scalebox{.4}{\includegraphics{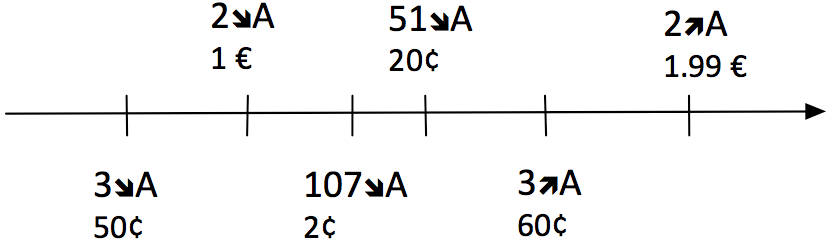}} \ec
\caption{Timeline illustrating the cost (in Euros) generated by an  approximation of  a $\DII$ set $A$ for a particular cost function.}\end{figure}

This definition, first given in \cite{Nies:book}, was conceived as an abstraction of  the  construction of  a c.e.\ noncomputable $K$-trivial set in Downey et al.~\cite{Downey.Hirschfeldt.ea:03}. Perhaps the  intuition  stems from analysis. For instance, the length of  a curve, i.e.\ a $\+ C^1$ function  $f \colon \, [0,1] \to \RR^n$, is given by $\int_0^1 || f'(t)|| dt$. The ``cost" of the change $f'(t)$ at stage $t$ is the velocity $|| f'(t)||$, and to have a finite total cost means that the curve is rectifiable. 

The paper~\cite{Nies:CalculusOfCostFunctions} also treats   non-monotonic cost functions, where   we define  $\limcost(x) = \liminf_s \cost(x,s)$ and otherwise retain the definition of the limit condition $\lim_x \limcost(x)= 0$.  Intuitively,   enumeration of $x$ into $A$   can only  take place    at a   stage  when the cost drops. This is reminiscent of  $\ES''$-constructions, for instance  building  a Turing  minimal pair of c.e.\ sets.  It would be interesting to define  a pair of   cost functions $\cost, \dost$ with the limit condition such that $A \models \cost$ and $ B \models \dost$ for c.e.\ sets $A,B$  imply that $A, B$ form a minimal pair. 

\subsection*{Applications of cost functions} Let  $\beta$ be a  left-c.e.\ real given  as $\beta = \sup_s \beta_s$ for a computable sequence $\seq{\beta_s}\sN s$ of rationals. We let $\cost_{ \beta }(x,s)= \beta_s - \beta_x$. Note that  $\cost_{ \beta }$ is a monotonic cost function with the limit condition. Modifying a result from \cite{Nies:AM}, in \cite{Nies:CalculusOfCostFunctions} it is shown that a $\DII$ set $A$ is $K$-trivial iff $A \models \cost_  \Omega$. Thus $\cost_  \Omega$ is a  cost function describing  $K$-triviality. This   raises the question whether obedience to cost functions stronger than $\cost_  \Omega$    can describe interesting subideals  of the ideal of $K$-trivial Turing degrees   (being a stronger  cost function means being harder to obey, i.e.\ more expensive).  

By the ``halves" of a set $Z$ we mean the sets $Z_0 = Z \cap \{2n \colon n \in \NN\}$ and $Z_1 = Z \cap \{2n+1 \colon n \in \NN\}$. If $Z$ is ML-random and $A \leT Z_0, Z_1$ then $A$ is a base for ML-randomness, and hence $K$-trivial. So we obtain a subclass $\+ B_{1/2}$ of the $K$-trivial sets, namely the  sets below both halves of a ML-random.  Bienvenu et al. \cite{Bienvenu.Greenberg.ea:16} had already proved that this subclass is proper.
Let $\cost_{\Omega, 1/2}(x,s) = \sqrt{\Omega_s - \Omega_x}$. In recent work, Greenberg, Miller and Nies obtained the following characterisation of $\+ B_{1/2}$.
\begin{thm}[\cite{Greenberg.Miller.ea:nd2}, Thm.\ 1.1. and its proof] The following are equivalent for a set $A$. 

\bi \item[(a)] $A$ is Turing below both halves of some ML-random
\item[(b)] $A$ is Turing below both halves of $\Omega$
\item[(c)] $A$ is a  $\DII$ set   that obeys $ \cost_{\Omega, 1/2}(x,s)$. \ei \end{thm} 
 
They generalise the result towards a characterisation of   classes $\+ B_{k/n}$, where $0< k < n$. The  class $\+ B_{k/n}$ consists of the $\DII $ sets $A$ that are Turing below the effective join of   any set of $k$ among the $n$-columns of some ML-random set $Z$; as before,  $Z$ can be taken to be  $\Omega$ without changing the class. The characterising cost function is $\cost_{\Omega, q}(x,s) = (\Omega_s - \Omega_x)^q$, where $q = k/n$. In particular, the class does not depend on the   representation of $q$  as a fraction of integers. By this cost function characterisation and the hierarchy theorem~\cite[Thm.\ 3.4]{Nies:CalculusOfCostFunctions},  $p<q $ implies that $\+ B_p$ is a proper subclass of $\+ B_q$.  

Following Hirschfeldt et al.~\cite{Hirschfeldt.Jockusch.ea:15}  we say  that a set  $A$ is \emph{robustly computable} from  a set $Z$ if $A \leT Y$ for each set $Y$ such that the symmetric difference of $Y$ and $Z$ has upper density $0$. In~\cite{Greenberg.Miller.ea:nd2} it is  shown that the union of all the $\+ B_q$, $q<1$ rational, coincides with the sets that are robustly computable from some ML-random set $Z$. 

\subsection*{Calibrating randomness notions via cost functions}

Bienvenu et al. \cite{Bienvenu.Greenberg.ea:16}  used cost functions to calibrate certain randomness notions.  
Let $\cost$ be  a monotonic cost function with the limit condition.  A descending  sequence $\seq{V_n}$ of uniformly c.e.\ open sets is a \emph{$\cost$-bounded test} if $\leb(V_n)  = O(\limcost(n))$ for all~$n$. We think of each $V_n$ as an approximation for $Y\in \bigcap_k V_k$. Being in  $\bigcap_n V_n$ can be viewed as a new sense of obeying $\cost$ that works  for  ML-random sets.  Unlike the  first  notion of obedience, here only the limit function $\ul \cost(x)$ is taken into account in the definition.

   Solovay completeness is a certain universal property of $\Omega$ among the left-c.e.\ reals;  see e.g.\ \cite{Downey.Hirschfeldt.ea:02}). Using this notion, one can show that the left-c.e.\ bounded  tests  defined above are essentially   the $\cost_\seq \Omega$-bounded tests. 

We now survey some related, as yet unpublished work of Greenberg, Miller, Nies and Turetsky dating from early 2015.
Hirschfeldt and Miller  in unpublished 2006 work had proven that for any $\Sigma_3$ null class $\+ C$  of  ML-random sets,  there is a c.e.\ incomputable set Turing below all the members of $\+ C$. Their argument can be recast in the language of cost functions in order to show the following (here and below $\cost $ is some monotonic cost function with the limit condition).
\begin{prop} \label{prop:basic fact}  Suppose that   $A \models \cost$ and $Y$ is in the $\SI 3$ null class of   ML-randoms captured by a $\cost$-bounded test. Then $A \leT Y$.  \end{prop}
We consider sets $A$ such that the converse   implication  holds as well.
\begin{definition} Let  $A$ be a $\DII$ set. We say that $A$ is \emph{smart  for $\cost$} if $A \models \cost$,  and $A \leT Y$ for each ML-random set $Y$
that   is captured by some  $\cost$-bounded test.  \end{definition}
Informally, $A$ is as complex as possible for obeying $\cost$, in the sense that  the only random sets $Y$ Turing above $A$ are the ones that are above $A$  because $A$ obeys the cost function showing that $A \leT Y$  via Proposition~\ref{prop:basic fact}.     

For instance, $A$ is smart for $\cost_  \Omega$ iff no   ML-random set $Y \ge_T A$  is   Oberwolfach random. 
Bienvenu et al. \cite{Bienvenu.Greenberg.ea:16}  proved that some $K$-trivial set $A$ is smart for $\cost_  \Omega$. This means that $A$ is the hardest to ``cover" by a ML-random: any ML-random computing $A$ will compute all  the $K$-trivials by virtue of not being Oberwolfach random.

 In the new work of Greenberg et al., this  result is generalised to arbitrary monotonic cost functions with the limit condition that imply $\cost_  \Omega$. %
\begin{theorem}[Greenberg et al., 2015] Let $\cost$ be  a monotonic cost function with the limit condition and suppose that    only $K$-trivial sets can obey~$\cost$. Then some  c.e.\ set $A$  is smart  for~$\cost$.  \end{theorem} 
 The proof of the more general result, available in \cite[Part 2]{LogicBlog:16}, is simpler than the original one. Since $A$ cannot be computable, the proof also yields a solution to Post's problem. This solution  certainly has no injury, because there are no requirements.

\section{The $\Gamma$ and the $\Delta$ parameter of a Turing degree}
 \label{s:gamma_delta}
 We proceed to our second method of gauging  the complexity of Turing degrees   with    methods inspired by analysis.  We will be able to give the intuitive notion of being ``close to computable" a metric interpretation.
 
For $Z\sub \NN$ the lower density is defined to be
$$\underline \eta (Z)  = \liminf_n \frac{|Z \cap [0, n)|}{n}.$$
(In the literature the symbol $\ul \rho$ is used. However, the same symbol denotes the Lebesgue density in the sense of Definition~\ref{def:LD}, so we prefer $\ul \eta$ here.) Hirschfeldt et al.\ \cite{Hirschfeldt:2015} defined the $\gamma$ parameter of a set $Y$:  $$\gamma(Y)=\sup \{ \underline \eta(Y\leftrightarrow S) \colon \, S \, \text{is computable}\}.$$ 
The $\Gamma$ operator was introduced by  Andrews, Cai, Diamondstone, Jockusch and Lempp \cite{Andrews.etal:2013}:
$$\Gamma(A) = \inf \{ \gamma(Y) \colon  \, Y \le_T A \}.$$
It is easy to see that this  only depends on the Turing degree of $A$: one can code $A$ back into $Y$ on a sparse computable set of positions (e.g.\ the powers of 2), without affecting $\gamma(Y)$.  

%$\overline  \eta (Z)  = \limsup_n \frac{|Z \cap [0, n]|}{n}$. $1 - \ul \eta (Y \lra R) = \ol \eta(Y \triangle R)$ introduces  a pseudo distance between sets. $1 -\gamma(Y)$ is the pseudo distance from $Y$ to the computable sets, and $1 -\Gamma(A)$ is the directed Hausdorff distance from the Turing cone below $A$ to the computable sets. This  decreases as  $A$ gets closer to being computable. 

We now provide dual concepts.
Let   $$\delta(Y)=\inf \{ \underline \eta(Y\leftrightarrow S) \colon \, S \text{ computable}\},$$
$$\Delta(A) = \sup \{ \delta(Y) \colon  \, Y \le_T A \}.$$

Intuitively,   $\Gamma(A)$ measures  how well computable sets can approximate    the sets that $A$ computes in the worst case (we take the infimum over all $Y \le_T A$). In contrast, $\Delta(A)$ measures how well the sets   that $A$ computes can approximate   the computable sets in the  best case (we take the supremum over all $Y \le_T A$). Note that $A \leT B$ implies $\Gamma(A) \ge \Gamma(B)$ and $\Delta(A) \le \Delta(B)$.

 It was shown in 
 \cite{Andrews.etal:2013} that $\Gamma(A) >1/2 \lra \Gamma(A) = 1 \lra A$ is computable. 
Clearly the maximum value of $\Delta(A)$ is $1/2$. It is attained, for example, when $A$ computes a Schnorr random set $Y$, because in that case  $\underline \eta(Y\leftrightarrow S)=1/2$ for each computable $S$. Merkle, Nies and Stephan have shown that $\Delta(A)=0$ for every $2$-generic $A$. 

\subsection*{Viewing $1-\Gamma(A)$ as  a Hausdorff pseudodistance}
For $Z\sub \NN$ the upper density is defined by $$\overline \eta (Z)  = \limsup_n \frac{|Z \cap [0, n)|}{n}.$$ For  $X,Y \in \cantor$ let $d(X,Y) = \ol \eta (X \triangle Y)$ be the upper density of the symmetric difference of $X$ and $Y$; this is clearly a pseudodistance  on Cantor space $\cantor$ (that is, two objects  may have  distance $0$ without being equal). For subsets $\+ U,\+ W$  of a pseudometric space $(M,d)$ recall the Hausdorff pseudodistance $$d_H(\+ U, \+ W) = \max  ( \sup_{u \in \+ U} d(u, \+ W),  \sup_{w \in \+ W} d(w, \+ U))$$ where $d(x, \+ R) = \inf_{r\in \+ R} d(x,r)$ for any $x \in M, \+ R \sub M$.  Clearly,  if $\+ U \supseteq \+ W$ then the second supremum is $0$, so that  we only need the first. 
\begin{figure}[htb]
%\bc \scalebox{.3}{\includegraphics{../figures/HDdistance.png}} \ec
\bc \scalebox{.3}{\includegraphics{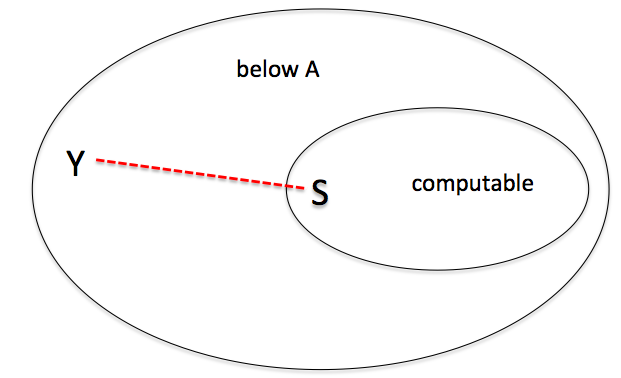}} \ec
\caption{Hausdorff pseudodistance $\sup_{Y\in \+ A}\inf_{S \in \+ R } d(Y,S)$.}
\end{figure}
The following fact, which is now clear from the definitions, states that $1-\Gamma(A)$ gauges how close  $A$ is to being computable, in the sense that it is  the Hausdorff distance between the   cone below $A$ and  the computable sets.
\begin{proposition} Given an oracle set  $A$ let $\+ A = \{ Y \colon \, Y \leT A\}$. Let $\+ R \sub \+ A$ denote the collection of    computable sets. We have
\[ 1- \Gamma(A)= d_H(\+ A, \+ R). \] \end{proposition}

To interpret $1-\Delta(A)$ metrically, we note that $1-\delta(Y) = \sup_{S\in \+ R} d(Y,S)$. So we can   view $1-\Delta(A) $ as a  one-sided ``dual" of the Hausdorff pseudodistance: \bc  $d^*_H(\+ A, \+ R) =   \inf_{Y  \in \+ A} \sup_{S\in \+ R} d(Y, S)$. \ec
For instance, for the  unit disc $D \sub \R^2$ we have $d^*_H(D,D) = 1$. 

%we view $\ul \eta (Y \lra R)$ as measuring the ``closeness" of $Y$ and $R$ (which \emph{increases} as the two sets get closer); $\delta(Y)$ is the closeness of $Y$ to the class of computable sets; $\Delta(A)$  measures the worst case of closeness in  the cone below $A$ to the computable sets. N
%%%%%%%%%
\subsection*{Analogs of  cardinal characteristics}  The operators  $\Gamma $ and $\Delta$ are  closely related to the analog in computability theory of  cardinal characteristics (see~\cite{Bartoszynski.Judah:book} for the background in set theory). Both the cardinal characteristics and their  analogs were  introduced   by Brendle and Nies  in the 2015 Logic Blog~\cite{LogicBlog:15}, building on the general framework of an analogy between set theory and computability theory set up by Rupprecht (\cite{Rupprecht:thesis}, also see \cite{Brendle.Brooke.ea:14}).  We only discuss the versions of the concepts in the setting of computability theory. 
\begin{definition}[Brendle and Nies]  For $p \in [0,1]$ let    \n $\+ D(\sim_p)$    be  the class of oracles $A$ that compute a set $X$ such that  $\gamma(X) \le p$, i.e.,   for each computable  set $S$, we have $\ul \eta (X \lra S) \le p$.  \end{definition} 
We note that by the definitions 
 $\Gamma(A) < p \RA A \in \+ D(\sim_p) \RA \Gamma(A) \le p$.   
 \begin{definition}[Brendle and Nies]  Dually, for $p \in [0,1/2)$ let \mbox{$\+ B(\sim_p)$}  be the class of oracles $A$ that compute a set $Y$  such that   for each computable  set $S$, we have $\ul \eta (S \lra Y) > p$.  \end{definition} 
 For each $p$  we have  
  $\Delta(A) > p \RA   A \in \+ B(\sim_p) \RA  \Delta(A) \ge p$. 
%%%%%%%%%%%%%%%%%%

\subsection*{Collapse of the $\+ D(\sim_p)$  hierarchy for  $p \neq 0$  after Monin}
 \begin{definition}  \label{df:ioe} For a computable function $h$, we let $\+ D(\neq^*_h)$, or sometimes $\+ D(\neq^*, h)$,  denote  the class of oracles $A$ that compute a function $x $  such that  $\ex^\infty n \, x(n)=  y(n)$   for each computable function $y < h$.   \end{definition}
 This highness notion of an oracle set $A$  was introduced by Monin and Nies in \cite{Monin.Nies:15}, where it was called ``$h$-infinitely often equal". The  notion also corresponds to a cardinal characteristic, namely $\frd(\neq^*_h)$ which is a bounded version of the well-known characteristic  $\frd(\neq^*)$.  The cardinal $\frd(\neq^*_h)$ is the least size of a set $G$ of $h$-bounded functions so that for each function $x$ there is a function $y$ in   $G$  such that    $\fa^\infty n [ x(n) \neq y(n)]$.    We note that $\+ D(\neq^*)$, i.e.\ the class obtained in Definition~\ref{df:ioe} when we omit the computable bound, coincides with having hyperimmune degree. See~\cite{Brendle.Brooke.ea:14} for background, and in particular for    motivation why  the defining condition for $\frd(\neq^*_h)$ looks like the negation of the condition for $\+ D(\neq^*_h)$.
 
 The proof of the  following fact provides a glimpse of the methods used to prove that the $\+ D(\sim_p)$  hierarchy collapses.
 \begin{prop} \label{prop: D} $\+ D(\neq^*, 2^{n!} )\sub \+ D(\sim_0)$.  \end{prop}
 \begin{proof} Suppose that $A \in \+ D(\neq^*, 2^{n!} )$ via a function $x \leT A$. Since $x(n) < 2^{n!} $ we can view $x(n)$ as a binary string of length $n!$. Let $L(x) \in \cantor$ be the concatenation of the strings $x(0), x(1), \ldots$, and let $X \leT A$ be the complement of $L(x)$. Given a computable set $S$, there is a computable function $y $ with $y(n) < 2^{n!}$ such that $L(y)= S$. Let $H(n)= \sum_{i < n} i!$. Since $x(n) = y(n)$ for infinitely many $n$, there are infinitely many intervals $[H(n), H(n+1))$ on which $X$ and $S$ disagree completely. Since $\lim_n n!/H(n)=0$ this implies $\ul \eta (X \lra S) =0$.  Hence $A \in \+ D(\sim_0)$. 
 \end{proof}
 We slightly paraphrase the main result of Monin's recent work \cite{LogicBlog:16}. It not only collapses the $\+ D(\sim_p)$  hierarchy, but also describes the resulting highness property combinatorially.

\begin{thm}[Monin]  $\+ D(p)= \+ D(\neq^*, \tp{(\tp n)})$  for each  $p\in (0,1/2)$. In particular, $\Gamma(A) < 1/2 \RA \Gamma(A) = 0$ so  only the values $0$ and $1/2$ can occur when  $\Gamma $ is evaluated on incomputable sets. \end{thm}
 The proof uses  the list decoding capacity theorem from the theory of error-correcting codes, which says that given a sufficiently large constant $L$,  a fairly large set of code words of a   length $n$ can be achieved if one allows that each word of length $n$ can be close (in the Hamming distance) to up to  $L$ of them. More precisely, independently of $n$, for each  positive $\beta <1$ there is $L \in \omega$ so that $\tp{\lfloor \beta n\rfloor}$ codewords can be achieved. (In the usual setting of error correction, one would have $L=1$, namely, each word is close to only one code word.)

 \subsection*{Collapse of the $\+ B(\sim_p)$ hierarchy for  $p \neq 0$  via a dual of Monin}
 
 \begin{definition}  For a computable function $h$, we let $\+ B(\neq^*_h)$ denote  the class of oracles $A$ that compute a function $y <  h$  such that   $\fa^\infty n \, x(n) \neq y(n)$  for each computable function $x$.   \end{definition}
    $\+ B(\neq^*)$, i.e.\ the class obtained   when we omit the computable bound, coincides with ``high or diagonally noncomputable"  (again see, e.g., \cite{Brendle.Brooke.ea:14}). As a dual to Proposition~\ref{prop: D} we have $ \+ B(\sim_0) \sub \+ B(\neq^*, 2^{n!} )$.
 
\begin{thm}[Nies]  $\+ B(\sim_p)= \+ B(\neq^*, \tp{(\tp n)})$  for each  $p\in (0,1/2)$. In particular, $\Delta(A) >0 \RA \Delta(A) = 1/2$ so there are only two possible $\Delta $ values. \end{thm}
For a proof see again \cite{LogicBlog:16}.

\section{Open questions} 
The development   we have sketched in Sections~\ref{s:analysis} and~\ref{s:LDT-K}    has led to two randomness notions. The first, density randomness, was born out of the  study of  randomness   via computable analysis. The second,   OW-randomness, was born out of the study of lowness via  randomness. We know that OW-randomness implies density randomness. 
 \begin{question} \label{qu:OW} Do OW-randomness and density randomness coincide? \end{question}
One direction of attack to answer this negatively could be  to look at other properties of points  that  are implied by OW-randomness, and show that density randomness does not suffice. By \cite[Thm.\ 6.1]{Miyabe.Nies.Zhang:15} OW-randomness of $z$ implies the existence of the limit in the sense of the Birkhoff ergodic theorem (Section~\ref{s:analysis})  for computable operators $T$ on a computable probability space $(\cantor, \mu)$,  and lower semicomputable functions $g\colon X \to \RR$. For another example, by \cite{Galicki.Nies:16} OW-randomness of $z$ also implies  an effective version  of the Borwein-Ditor theorem: if $\seq {r_i}\sN i$ is a computable null sequence of reals and $z\in \+ C$ for a $\PPI$ set $\+ C\sub \R$, then $z+r_i \in \+ C$ for infinitely many $i$. 

Lowness for density randomness coincides with $K$-triviality by \cite[Thm. 2.6]{Miyabe.Nies.Zhang:15}. Lowness for   OW randomness is merely known to imply $K$-triviality  for the reasons discussed in Section~\ref{s:redness}; further, an incomputable c.e.\ set that is low for OW-randomness has been constructed in unpublished worked  with Turetsky.

\begin{question} Characterise lowness for OW-randomness. Is it the same as $K$-triviality? \end{question}

Section~\ref{s:gamma_delta} leaves open several questions.

\begin{question} Is $\+ D(\sim_0)$ a proper subclass of $\+ D(\neq^*, \tp{(\tp n)})= \+ D(1/4)$?

\n  Is $\+ D(\neq^*, \tp{n!})$ a proper subclass of $\+ D(\neq^*, \tp{(\tp n)})$?\end{question}
By Proposition~\ref{prop: D} an affirmative answer to the first part implies an affirmative answer to the second.  The dual open questions are:

\begin{question} Is $\+ B(\neq^*, \tp{(\tp n)})= \+ B(\sim_{0.25})$ a proper subclass of $\+ B(\sim_0)$? 

\n Is  it a proper subclass of $\+ B(\neq^*, \tp{(n!)})$?\end{question}

\subsection*{Acknowledgement} Most of the research surveyed in  this article was supported by the Marsden Fund of New Zealand.

\def\cprime{$'$}

%\bibliographystyle{plain}
% 
%\bibliography{../../Logicsharing/bibs/Nies,../../Logicsharing/bibs/randomness,../../Logicsharing/bibs/settheory,../../Logicsharing/bibs/various,../../Logicsharing/bibs/recursiontheory,../../Logicsharing/bibs/analysis,../../Logicsharing/bibs/kucera,../../Logicsharing/bibs/modeltheory}
% 
\end{document}